\documentclass{amsart}

\usepackage[all,RR,hyperref]{bi-discrete}
\newcommand{\btimes}{\mathbin{\rotatebox[origin=c]{90}{$\Join$}}}

\usepackage[backend=biber,maxbibnames=6,maxalphanames=5,style=alphabetic,bibencoding=utf8,giveninits,url=false,isbn=false, maxcitenames=5,
mincitenames=5]{biblatex}
\AtBeginBibliography{\small}
\DeclareMathOperator{\dis}{dist}

\newtheorem*{theorem*}{Theorem}

\usepackage{mathtools}
\usepackage{tabls}
\usepackage{pgfplots}

\usepackage{tikz}
\usetikzlibrary{intersections,calc,arrows.meta,patterns,angles,quotes,shapes.geometric}
\usepackage{tikz}

\providecommand{\Rd}{\mathbb{R}^d}
\providecommand{\ind}{\mathrm{index}}

\newcommand{\vfullpattern}[1]{%
  #1
  
  \vrepeatpatternthree{#1}
  
  \vrepeatpatternbig{#1}
}

\newcommand\vrepeatpatternbig[1]{%
  \begin{scope}[shift={(0,-1)},scale=3]
    #1
  \end{scope}
  \begin{scope}[shift={(0,1)},scale=3]
    #1
  \end{scope}
}

\newcommand{\vrepeatpatternone}[1]{%
  \begin{scope}[shift={(0,-1/3)},scale=1/3]
    #1
  \end{scope}
  \begin{scope}[shift={(0,1/3)},scale=1/3]
    #1
  \end{scope}
}

\newcommand{\vrepeatpatterntwo}[1]{%
  \begin{scope}[shift={(0,-1/3)},scale=1/3]
    #1
    \vrepeatpatternone{#1}
  \end{scope}
  \begin{scope}[shift={(0,1/3)},scale=1/3]
    #1
    \vrepeatpatternone{#1}
  \end{scope}
}

\newcommand{\vrepeatpatternthree}[1]{%
  \begin{scope}[shift={(0,-1/3)},scale=1/3]
    #1
    \vrepeatpatterntwo{#1}
  \end{scope}
  \begin{scope}[shift={(0,1/3)},scale=1/3]
    #1
    \vrepeatpatterntwo{#1}
  \end{scope}
}

\begin{document}

\title{Nonlocal and mixed models with Lavrentiev gap}

\author{Anna  Balci}
\address{Anna~Balci, University Bielefeld, Universit\"atsstrasse 25, 33615
  Bielefeld, Germany.}
\email{akhripun@math.uni-bielefeld.de}
\thanks{This  research is funded  by the Deutsche Forschungsgemeinschaft (DFG, German Research Foundation) - SFB 1283/2 2021 - 317210226.}
\begin{abstract}
 We present  a general framework for constructing examples on  Lavrentiev energy gap for nonlocal problems and  apply it to several nonlocal and mixed  models of double-phase type.
\end{abstract}

\maketitle


\section{Introduction}

The Lavrentiev gap  is known to appear in the study of local problems of variational type. The minimum of the integral functional  taken over smooth functions may differ from the one taken over the associated energy space. This leads to significant challenges such as special kind of non-uniqueness for such problems, non-density of smooth functions,  lack of regularity.     Within this article we  introduce the procedure to construct  examples of an energy gap for nonlocal and mixed  problems.  
In particular we concentrate on several nonlocal and mixed models of double-phase type.  The energy is defined as the sum of two parts ("phases") ~$\mathcal{J}^s_p$ and~$\mathcal{J}^t_{q,a}$ which are connected through a possible degenerate coefficient~$a(\cdot,\cdot)$ in the energy~$\mathcal{J}^t_{q,a}$.  The strong part~$\mathcal{J}^t_{q,a}$ is only activated at the part where the weight~$a(\cdot,\cdot)$ does not vanish.  
Depending on the concrete structure of the phases we distinguish the following cases 
\begin{align*}
     &\text{Local model \textbf{(I)}  :\,}\\&\quad  \mathcal{J}(v)= \int_\Omega \frac 1p \abs{\nabla v}^p\, dx+\int_\Omega \frac 1 q a(x)\abs{\nabla v}^q \, dx,\\
     &\text{Nonlocal-local model \textbf{(II)}:\,}\\&\quad  \mathcal{J}(v)= \int_{\setR^d} \int_{\setR^d} \frac 1 p \frac{\abs{v(x)-v(y)}^p}{\abs{x-y}^{d+sp}} dxdy+\int_\Omega \frac 1 q a(x)\abs{\nabla v}^q\, dx,\\
     &\text{Local-nonlocal model \textbf{(III)}:\,}\\&\quad  \mathcal{J}(v)= \int_\Omega \frac 1p \abs{\nabla v}^p\,dx+ \int_{\setR^d} \int_{\setR^d} \frac 1 qa(x,y) \frac{\abs{v(x)-v(y)}^q}{\abs{x-y}^{d+tq}} dxdy,\\    
     &\text{Nonlocal model}  \textbf{ (IV)}:\\ 
     &\quad  \mathcal{J}(v)= \int_{\setR^d} \int_{\setR^d} \frac 1 p \frac{\abs{v(x)-v(y)}^p}{\abs{x-y}^{d+sp}} dxdy+\int_{\setR^d} \int_{{\setR^d}} \frac 1 qa(x,y) \frac{\abs{v(x)-v(y)}^q}{\abs{x-y}^{d+tq}} dxdy.
\end{align*}
Here~$0<s\le t<1$,~$p,q\in (1,\infty)$ and~$a:\setR^d\times\setR^d\to\setR$ is a nonnegative coefficient.  When~$a(x,y)=0$, the double phase models reduce to the~$p$-Laplacian for the models~\textbf{I} and~\textbf{III}, and to the~$s$-fractional~$p$-Laplacian for the models~\textbf{II} and~\textbf{IV}.  

\subsection{Known models}
Most classical and well-studied is the local model \textbf{I}.   The  integrand~$$\Phi(x,t)=\frac 1 p \abs{t}^p+a(x)\frac 1 q \abs{t}^q$$  defines the corresponding generalized Sobolev-Orlicz spaces~$W^{1,\Phi(\cdot)}(\Omega)$ (the natural energy space for $\mathcal{J}$ for  model \textbf{I}) . The Lavrentiev gap in this case is the inequality  
\begin{equation}\label{Lavr}
 \inf  \mathcal{J}(W^{1,\Phi(\cdot)}(\Omega)) < \inf  \mathcal{J}(C_0^\infty(\Omega)).
\end{equation}

 A similar phenomenon for boundary value problems can be expressed as the inequality
\begin{equation}\label{Lavr_bvp}
 \inf  \mathcal{J}\big((u_D+W^{1,\Phi(\cdot)}_0(\Omega)\big) < \inf  \mathcal{J}\big(u_D+C_0^\infty(\Omega)\big)
\end{equation}
for some $u_D\in C^1(\overline{\Omega})$.

A closely related problem is density of smooth functions in the natural energy space of the functional. Denote the closure of smooth forms from $W^{1,\Phi(\cdot)}(\Omega)$ in this space by $H^{1,\Phi(\cdot)}(\Omega)$. If any function from the domain of $\mathcal{J}$ can be approximated by smooth functions with energy convergence (equivalently, if $H^{1,\Phi(\cdot)}(\Omega)=W^{1,\Phi(\cdot)}(\Omega)$, which is abbreviated to $H=W$) then the Lavrentiev gap is obviously absent. In the autonomous case, when  the integrand $\Phi=\Phi(t)$ is an Orlicz function independent of $x$, the Lavrentiev phenomenon is absent ($H=W$). When~$\Phi:=\Phi(x,t)$ and~$d=1$ it was first  observed in \cite{Lav27}; for more general integrands see also recent results in \cite{Mar23}. 

In 1995 Zhikov~\cite[Example~3.1]{Zhi95} considered local model~\textbf{I} in the context of Lavrentiev phenomenon. He constructed with checkerboard
setup a weight~$a \in C^{0,\alpha}(\overline{\Omega})$ with
$\alpha=1$, $\Omega= (-1,1)^2$ and $p < 2 < 2+\alpha =3 < q$ such that
the Lavrentiev gap occurs. This example was generalized by 
~\textcite{EspLeoMin04} to the case of higher dimensions and less
regular weights, i.e. $\alpha \in (0,1]$. In particular,
for~$\Omega=(-1,1)^d$, they constructed a
weight~$a \in C^{0,\alpha}(\overline{\Omega})$ and exponents
$1 < p < d < d+ \alpha < q$ such that the Lavrentiev gap occurs.   In both examples by Zhikov and
Esposito-Leonetti-Mingione the two exponents~$p$ and~$q$ cross the
dimensional threshold~$d$. \textcite{FonMalMin04} studied fractal singular sets for such problems.  The general procedure to construct examples on Lavrentiev gap for the local case without any dimensional restrictions was presented by \textcite{BalDieSur20} and was based on  fractal contact sets. For the so-called borderline case of double-phase the results on Lavrentiev gap are obtained by \textcite{BalSur21}. Recently this procedure was generalised to the case of the variational problems with differential forms by \textcite{BalSur23}.

Let us mention that  Lavrentiev gap could appear for vectorial autonomous problems, where integrand only depends on the gradient of the function.  Notable examples of such scenarios emerge in the field of nonlinear elasticity and were  constructed by \textcite{FHM03} and more very recently modified by~\textcite{Almi2023ANE}.

Challenges also arise in the numerical approximation of problems exhibiting the Lavrentiev phenomenon. Discussions on these challenges for non-autonomous functionals can be found in the works of \textcite{BalOrtSto22} and \textcite{BalKal23}. 

Lavrentiev gap  for the double phase potential can also be seen as a
lack of higher regularity, see ~\textcite{Mar89} for the first
example in this direction. In fact, local minimizers of~$\mathcal{F}$
need not be~$W^{1,q}$-functions unless~$a$, $p$ and~$q$ satisfy
certain assumptions. In fact, if $\frac{q}{p} \leq 1+\frac \alpha{d}$
and $a \in C^{0,\alpha}$, then minimizers of~$\mathcal{F}$ are
automatically in~$W^{1,q}$, see \textcite{ColMin15}. Moreover, bounded
minimizers of~$\mathcal{F}$ are automatically~$W^{1,q}$
if~$a \in C^{0,\alpha}$ and~$q \leq p+\alpha$,
see~\textcite{BarColMin18}. If the minimizer is from~$C^{0,\gamma}$, then
the requirement can be relaxed to $q \leq p + \frac{\alpha}{1-\gamma}$, see 
\textcite[Theorem~1.4]{BarColMin18}. 

The different question is related to the density of smooth functions in the corresponding generalized Sobolev-Orlicz spaces. The sufficient conditions for density  for double phase local model were formulated in \cite{Zhi95}[Lemma 2.1] for the case of Lipschitz continuous weight~$a(x)$; the approach based on boundedness of convolution for different non-autonomous models is contained in the book by~\textcite{HarHas19}, the condition for the density in this case is~$\frac{p}{q}\le 1+\frac{\alpha}{d}$. 
The questions of  density of smooth functions was recently studied by \textcite{BulGwiSkr22}, where they  established an improved range of exponents $p$ and  $q$ for which the Lavrentiev phenomenon does not occur;   see also work  by  \textcite{BCDM23} with a refinement of this method for~$\alpha>1$.  Recently, \textcite{DeFilMin23} obtained regularity results  for double phase problems at nearly linear growth.

In a nonlocal context, \textcite[Section 6]{BBCK09} employed a class of symmetric jump processes to create bounded harmonic functions that lack continuity. A careful analysis of the  geometrical configuration of this example shows that it  closely resembles the construction of a one-saddle point. In this particular model, the generative integrand is linear; however, the presence of a singular kernel introduces an intriguing phenomenon. This phenomenon prevents the Harnack inequality from conferring regularity properties.  The study of non-local nonlinear models goes back to \textcite{KuuMinSi15_1, KuuMinSir15} for the fractional~$p$-Laplacian. The nonlocal double-phase \textbf{IV} was introduced by \textcite{DeFPal19}, where they studied Hölder regularity of minimizers under the assumptions~$q\ge p$.   Regularity estimates for local minimizers of nonlocal functionals with non-standard growth of (p,q)-type were studied by  \textcite{ChaMinWei22}; Self-improving inequalities for bounded weak solutions to nonlocal double phase equations were obtained by \textcite{ScoMen22}.  In paper by  \textcite{BOS22} proved  local boundedness and Hölder continuity for weak solutions to nonlocal double phase problems for~$0<s\le t<1<p\le q<\infty$ and~$p\le d$; \textcite{ChaKimWei23} obtained a full Harnack inequality for local minimizers  and weak solutions  to nonlocal problems with non-standard growth. Local {H}\"{o}lder regularity for nonlocal equations with variable powers was studied by~\textcite{Ok23}.   Let us also mention recent significant  contributions to the regularity theory for nonlocal models with coefficients by \textcite{Now23,Now23_1}, \textcite{KuuNowSir23},   \textcite{DieNow23} for Calder\'on-Zygmund estimates for weighted nonlocal~$p$-Laplacian. 

The theory of function spaces connected to the nonlocal problems is based on the theory of  classical Orlicz spaces, see for example \textcite[Section 3]{ACPS21}  for fractional Sobolev-Orlicz spaces.

Mixed problems of~$p,s$-Laplacian type were studied by~\textcite{DeFilMin22a}. Mixed nonlocal-local  problems of the type \textbf{II} with the     were considered by \textcite{BuyLeeSong23}, where they obtained Hölder regularity results for such problems under the assumption~$s\in (0,1)$ and~$1<p\le q$.

To our knowledge the mixed local-nonlocal problem of the type~\textbf{III}, where the~$J^t_{q,a}$ phase is nonlocal, was not  studied till now.

\subsection{New models and main results} In this paper we show that nonlocal and mixed  problems of the type~\textbf{I}-\textbf{IV} admit energy gaps. We present the general method of constructing such examples for mixed and nonlocal models. We introduce the mixed model~\textbf{III} and present the sharp bounds for the energy gap to occur.    Our method extends and strengthens  our previous local  approach for the model \textbf{I} from \textcite{BalDieSur20} and \textcite{BalSur21}.  
This  contains a new approach based on estimates for operator of Riesz potential type for nonlocal problems, which makes it possible to avoid  the duality theory, which in the general case is not available for the nonlocal case. 

Now we state the main results of this paper. We work with three models: nonlocal model \textbf{IV}, mixed models \textbf{II}-\textbf{III}. In each of these scenarios, we construct illustrative examples of the Lavrentiev gap. Importantly, the methodology we propose is not confined solely to these models and can be extended to a wider range of problem types. The main result if formulated in the following theorem.
\begin{theorem}[Energy gap]\label{theoremA}
 Let~$\Omega=(-1,1)^d$ with~$d\ge 2$,~be such that~$t-\frac{q}{d}>s-\frac{p}{d}$ and~$t>s$.    Then there exists a weight~$a(\cdot,\cdot):\setR^d\times \setR^d\to [0,\infty)$ which is measurable, symmetric, bounded  and Hölder continuous function in $C^\alpha\cap L^\infty (\setR^d\times\setR^d)$ and  the boundary value $u_D\in C^\infty (\overline{\Omega})$ such that 
 \begin{align*}
  \inf  \mathcal{J}\big(W_{u_D}(\Omega)\big) < \inf  \mathcal{J}\big(H_{u_D}(\Omega)\big)
 \end{align*}
 holds.
\end{theorem}

Function spaces~$W_{u_D}(\Omega)$ and $H_{u_D}(\Omega)$ are defined later in Section~\ref{sec:energy}.  
\begin{remark}
 Let us mention, that the Theorem~\ref{theoremA} gives the example of the energy gap for the case~$q<p$ for the model~\textbf{IV}. In the purely local context this effect does not appear, since if~$q<p$ the situation is similar to the~$p$-Dirichlet energy, where Lavrentiev gap does not occur.  
\end{remark}
\bigskip
We collect all sets of parameters witch ensures the appearance of energy gap  for local, nonlocal and mixed problems in  the following table:
\begin{table}[!htbp]
 \begin{center}
\begin{tabular}{ c c c }
 model & subcritical case~$s-\frac dp\le 0$ & supercritical case~$s-\frac d p>0$  \\ 
 \bigskip
 \textbf{I} & $q>p+\alpha$ &    $\frac{q}{q-1}\bigg(1-\frac{d+\alpha}{q} \bigg) >\frac{p}{p-1}\bigg(1-\frac{d}{p}\bigg)$\\
 \bigskip
 \textbf{II} &$q>sp+\alpha$ & $\frac{q}{q-1}\bigg(1-\frac{d+\alpha}{q} \bigg)>\frac{p}{p-1}\bigg(s-\frac{d}{p}\bigg)$ \\ 
 \bigskip
 \textbf{III} &  $tq>s+\alpha$ & $\frac{q}{q-1}\bigg(t-\frac{d+\alpha}{q} \bigg)>\frac{p}{p-1}\bigg(1-\frac{d}{p}\bigg)$ \\ 
 \bigskip
 \textbf{IV} & $tq>sp+\alpha$ & $\frac{q}{q-1}\bigg(t-\frac{d+\alpha}{q} \bigg)>\frac{p}{p-1}\bigg(s-\frac{d}{p}\bigg)$ \\   
 \end{tabular}
\end{center}
\caption{Energy gap parameters for different models}
\end{table}

Note, that the condition for the supercritical case and  local model~\textbf{I} here are written in the form to be easy compared with mixed and nonlocal case. This condition the same as  the original one for the local double phase,  formulated in \cite{BalDieSur20}[Section 4.2]  as  $$q>p+\alpha \frac{p-1}{d-1}.$$

 Let us mention  that for the supercritical case  no higher regularity results are knows till now neither for the nonlocal no for mixed problems. Nevertheless our conditions for the supercritical case  indicate the possible bounds for the model parameters in future investigations, particularly in the density of smooth functions and regularity aspects.

 For both subcritical and supercritical scenarios in the mixed and nonlocal models \textbf{II}-\textbf{IV}, the conditions we derive for the Lavrentiev gap \textbf{do not necessitate}~$q>p$. Conversely, in the local case of model~\textbf{I}, such a condition becomes imperative to substantiate the "double-phase" structure and the existence of the Lavrentiev gap. Despite the absence of regularity theory for the case~$q<p$, the condition~
 $$
 t-\frac{d}{q}>s-\frac d p
 $$
 ensures the embedding~$W^{t,q}\embedding W^{s,p}$, thereby preserving the double-phase structure.

The proof of Theorem~\ref{theoremA} relies on the construction and analysis of a special function, which is referred to as the "competitor" and denoted by $u_C=u(x, y)$. This function is designed to possess certain properties and belongs to the energy space of the double-phase problem, which ensures that it has finite energy. However, for the subcritical case~$\ind(W^{s,p})<0$ the competitor $u_C$ has  discontinuities across the fractal contact set (we  denote the smooth version of it as~$u_D$ to be a boundary value).  The transition weight function   $a(x, y)$  is constructed in a specific manner to suit the requirements of the problem and its parameters. In particular it behaves like~$\abs{x_d}^\alpha+\abs{y_d}^\alpha$ in the upper and lower regions of the domain~$\Omega$, and it vanishing in the other regions close to the fractal contact set. In addition we construct a function~$u_D$ that agrees with the competitor~$u_C$ on~$\Omega^\complement$ but is~$C^\infty$. This function~$u_D$ will be used later as the boundary value. 

The overall strategy of the proof involves showing, under certain additional assumptions on the parameters of the model, that are different for the cases where $\ind (W^{s,p}) > 0$ and $\ind (W^{s,p}) < 0$, that the following  desired properties are achieved:
\begin{enumerate}
 \item ~$\mathcal{J}^s_p(u_c)<\infty$ and ~$J_{q,a}^t(u_C)=0$ (the competitor  does not see the ~$J_{q,a}^t$ phase). 
 \item for any smooth function~$v\in C^0(\overline\Omega)$ with~$v=u_D$ on~$ \Omega^{\complement}$~the second phase~$J_{q,a}^t$ is bounded from below:\quad $$J_{q,a}^t(v)\ge \const>0.$$ 
\end{enumerate}

Combination of these properties gives an example of energy gap.

\subsection{Outline}
 In Section~\ref{sec:energy} we recall some basic definitions, introduce corresponding fractional Sobolev-Orlicz  energy spaces and existence results. In Section \ref{sec:Cantor} we describe Cantor barriers.  In Section \ref{sec:sub} we describe the general framework for the construction of examples for the subcritical case~$\ind(W^{s,p})<0$ and in  Section \ref{sec:super} for the supercritical case $\ind(W^{s,p})>0$.

\section{Energy and variational formulation}\label{sec:energy}

In this section we define the energies and corresponding energy spaces. 

We assume that~$\Omega \subset \Rd$ is a Lipschitz domain of finite
measure.  Later in our
applications we will only use~$\Omega= (-1,1)^d$. By~$\lambda \Omega$ we denote ~$\lambda$-times enlarged~$\Omega$.  By $B^m_r(x)$ we denote the ball
of~$\setR^m$ with radius~$r$ and center~$x$. We denote
by~$\indicator_A$ the indicator function of the
set~$A$. By~$L^p(\Omega)$ and $W^{s,p}(\Omega)$ we denote the usual
Lebesgue and fractional Sobolev  spaces. Moreover, let~$W^{s,p}_0(\Omega)$ be the
Sobolev space with zero boundary values. The spaces we deal with could be seen as generalisations of the fractional Sobolev spaces~$W^{s,p}(\Omega)$.  Similar to Sobolev-Orlicz spaces our spaces will be defined by a modular, in most cases a non-local one. 
We
use~$c>0$ for a generic constants whose value may change from line to
line but does not depend on critical parameters. We also abbreviate
$f \lesssim g$ for $f \leq c\, g$.

To study minimizers of the problems~\textbf{II}-\textbf{IV} we need to introduce relevant function spaces. For simplicity we define the energy spaces for the nonlocal model~\textbf{IV}. The mixed energies~\textbf{II} and ~\textbf{III} could be treated in a similar way. We assume that transition weight~$a(x,y)$  vanishes outside of~$3\Omega$. In our examples later  we take a particular competitor  function~$u_D\in L^\infty (\setR^d)$ as a boundary value such as~$u_D$ vanishes outside~$3\Omega$.  In general it could happen, that the boundary  function has infinite  energy~$\mathcal{J}(v)$. 
 The truncation outside~$3\Omega$ allows us  to avoid the renormalized  energies approach, which was presented in  \textcite{CMY17};~\textcite{DPV19}.   

\begin{definition}[Energy spaces~$W_0$,~$W_{u_D}$]\label{def:W} 
The generalized fractional  Sobolev space $W_0(\Omega)$,
~$W_{u_D}(\Omega)$ with   $p,q\in[1,\infty)$ and~$s,t\in (0,1)$ is the space of all measurable functions such as  
\begin{align*}
W_0(\Omega)&:=\set{v\in L^1_{\loc}(\setR^d):\, \mathcal{J}(v)<\infty\text{ with } v|_{ \Omega^\complement}=0},\\
W_{u_D}(\Omega)&:=u_D+W_0(\Omega),
\end{align*}
where~$u_D\in L^\infty(\setR^d)$ and~$u_D|_{(3\Omega)^\complement}=0$.
\end{definition}

We abbreviate for simplicity~$W_0=W_0(\Omega)$. Now~$W_0$ is a Banach spaces equipped with the norm

\begin{align*}
 \norm{v}_{\mathcal{J}}:=\inf\biggset{\lambda>0:\mathcal{J}\bigg(\frac{v}{\lambda}\bigg) \le 1}.
\end{align*}
In general smooth functions are not dense in~$W$, so the following spaces are defined by closure of smooth functions in the energy norm

\begin{definition}[Energy spaces~$H_0$,~$H_{u_D}$]\label{def:H} 
The fractional Sobolev space~$H_0$ is defined as the closure of~$C_0^\infty$-functions in ~$W$:
\begin{align*}
H_0&:=\overline{C_0^\infty(\setR^d)}^{\norm{\cdot}_{\mathcal{J}}},\\
H_{u_D}&:={u}_D+H_0(\Omega),
\end{align*}
where~$u_C\in L^\infty(\setR^d)$ and~${u}_D=\eta u_C$, with cut-off function $\eta \in C^\infty_0(\Omega)$,  $\indicator_{(-\frac46,\frac46)^d} \leq \eta \leq
  \indicator_{(-\frac56,\frac56)^d}$ and $\norm{\nabla \eta}_\infty \leq
  c$.
\end{definition}

Note that~$H_{u_D}$ is a closed affine subspace of~$W_0(3\Omega)$.
\begin{lemma}
For~$u_D$,~$W$,~$W_{u_D}$ from Definition~\ref{def:W} there hold
 \begin{align*}
    W_{u_D}(\Omega) \subset W_0(3\Omega).
 \end{align*}

\end{lemma}
\begin{proof}
    Let~$v=u_D+w$, where~$w\in W_0(\Omega)$.\\
    If~$x,y\in \Omega^\complement$, then~$\mathcal{J}(v)=0$ and the proof is complete.\\
    If~$x,y\in 2\Omega$, then since~$v, u_D\in W$ and $\mathcal{J}(v)<\infty$.\\
     Let~$x\in \Omega$ and~$y\in (2B)^\complement$, then $\mathcal{J}(v)<\infty$ since~$u_D=0$ outside of~$3\Omega$.
\end{proof}

\noindent ({\it W-minimization})  Minimize $\mathcal{J}$ over the set $W_{u_D}$ from Definition~\ref{def:W}
\begin{align}\label{eq:var1}
 \mathcal{J}(v)  \rightarrow \min, \quad v\in W_{u_D}.
 \end{align}
\\
\noindent ({\it H-minimization}) Minimize $\mathcal{J}$ over the set $H_{u_D}$ from Definition~\ref{def:H}

\begin{align}\label{eq:var2}
 \mathcal{J}(v)  \rightarrow \min, \quad v\in H_{u_D}.
 \end{align}

\begin{theorem}[Existence of~$W$ and~$H$-solutions]
 The variational problems \eqref{eq:var1} and \eqref{eq:var2} have a unique minimizers: $$\omega_W \in W\text{ and  } \omega_H \in H.$$ \end{theorem}

The functional~$\mathcal{J}$ is well defined on~$W_{u_D}$ and ~$H_{u_D}$, uniformly convex and coercive, so  we can apply the standard theory of variational integrals, which ensures the existence of a unique minimizer  ~$v\in W_{u_D}$ and~$v\in H_{u_D}$.

\section{Cantor barriers}\label{sec:Cantor}

We start  with the one dimensional generalized Cantor
set~$\frC_\lambda$ with $\lambda \in (0,\frac 12)$, which is also
known as the (1-2$\lambda$)-middle Cantor set. We begin with the
interval~$\frC_{\lambda,0} := (-\frac 12,\frac 12)$. Then we
define~$\frC_{\lambda,k+1}$ inductively by removing the
middle~$1-2\lambda$ parts from~$\frC_{\lambda,k}$. In particular, we
define~$\frC_\lambda := \cap_{k \geq 1} \frC_{\lambda,k}$.  The
corresponding Cantor measure~$\mu_\lambda$ (also Cantor distribution)
is then defined as the weak limit of the
measures~$\mu_{\lambda,k}:=(2\lambda)^{-k}
\indicator_{\frC_{\lambda,k}}\,dx$. The factor $(2 \lambda)^{-k}$ is chosen such 
that~$\mu_{\lambda,k}([-\frac 12,\frac 12])=1= \mu([-\frac 12,\frac
12])$.  Thus, $\mu_\lambda(\setR)=1$ and
$\support \mu_\lambda = \frC_\lambda$.  The fractal dimension
of~$\frC_\lambda$ is
$\dim(\frC_\lambda) = \log(2)/ \log(1/\lambda) \in (0,1)$,
i.e. $\lambda = 2^{-\frD}$.

We will also need the~$m$-dimensional Cantor sets~$\frC_\lambda^m$ and
its distribution~$\mu^m_\lambda$, which are just the Cartesian
products of $\frC_\lambda$ and $\mu_\lambda$. Its fractal dimension
is~$\dim \frC_\lambda^m = m \dim \frC_\lambda = m \log(2)/ \log(1/
\lambda) \in (0,m)$. Note that
$\frC_\lambda^m = \cap_{k \geq 1} \frC_{\lambda,k}^m$.

Our construction is based on the following "upper" and "lower" sets: 
\begin{align*}
    \mathcal{M}_\tau^+&:={\set{z\in \Omega, z_d>0: \dis(\overline{z},\frC^m_\lambda)\le \tau\abs{z_d}}},\\
    \mathcal{M}_\tau^-&:={\set{z\in \Omega, z_d<0: \dis(\overline{z},\frC^m_\lambda)\le \tau\abs{z_d}}},\\
    \mathcal{M}_\tau&:=\mathcal{M}_\tau^+ \cup \mathcal{M}_\tau^-.
\end{align*}

\begin{lemma}[\cite{BalDieSur20}]
  \label{lem:cantor-estimates}
  Let~$\lambda \in (0,\frac 12)$, $1\leq m \leq d$ and
  $\frD := \dim(\frC_\lambda^m) = -m \log(2)/\log(\lambda)$. We
  use the notation $x=(\overline{x},\hat{x}) \in \setR^m \times \setR^{d-m}$.
  Then we have the following properties: 
  \begin{enumerate}
    \setlength{\itemsep}{1ex}%
  \item \label{itm:cantor-estimates1} For  every
    ball~$B^m_r(\overline{x})$ there
    holds~$\mu^m_\lambda (B^m_r(\overline{x})) \lesssim
    \indicator_{\set{d(\overline{x},\frC^m_\lambda)\le r}} r^{\frD}$.

  \item \label{itm:cantor-estimates2} For all $r>0$ there holds
    $
    \mathcal{L}^m\big(                                \set{
      \overline{x}:   
      d(\overline{x},\frC^m_\lambda) \leq  r } \big)
    \lesssim r^{m-\frD}$.
  \item \label{itm:cantor-estimates3}
    For all $\tau \in (0,4]$ there
    holds
    \begin{align*}
      \bigabs{\big((\mu_\lambda^m \times \delta_0^{d-m}) *
      \indicator_{\set{\abs{\overline{x}}\leq \tau\abs{\hat{x}}}}\big)(x)} 
      \lesssim
      \indicator_{ \set{ 
      d(\overline{x},\frC^m_\lambda) \leq\tau \abs{\hat{x}}}}(x)
      \,\abs{\hat{x}}^\frD.
    \end{align*}
  \end{enumerate}
\end{lemma}

We also use  the following slightly reformulated version of    \cite[Lemma 7]{BalDieSur20}.
\begin{lemma}
  \label{lem:basic-integrals}
  Let~$\lambda \in (0,\frac 12)$, $1\leq m \leq d$ and
  \begin{align*}
    \frD := \dim(\frC_\lambda^m) = -m \log(2)/\log(\lambda).
  \end{align*}
  
  We use
  the notation $x=(\overline{x},\hat{x}) \in \setR^m \times \setR^{d-m}$. Then
  \begin{enumerate}
    \setlength{\itemsep}{1ex}
  \item \label{itm:cantor-weak1}
    $\abs{\hat{x}}^{-(d-\frD)} \indicator_{\set{\abs{\overline{x}} \leq 4 \abs{\hat{x}}}} \in
    L^{1,\infty}(\Rd)$.
  \item \label{itm:cantor-weak2}
    $\abs{\hat{x}}^{-(d-\frD)} \indicator_{\set{\dis(\overline{x},\frC^m_\lambda)
        \leq 4\, \abs{\hat{x}}}} \in L^{1,\infty}(\Rd)$.
  \end{enumerate}
\end{lemma}
\begin{corollary}\label{cor:cor}
 If~$\sigma>-(d-\frD)$, then
 \begin{align*}
  \int_\Omega \abs{\hat{z}}^\sigma\indicator_{\mathcal{M}_4(z) }\,dz<\infty.
 \end{align*}

\end{corollary}
\section{Nonlocal model \textbf{IV}}\label{nonlocal}
In this section we construct example of Lavrentiev phenomenon for the nonlocal model \textbf{IV}, so when both phases $\mathcal{J}_{p}^s(v)$ and~$\mathcal{J}_{q,a}^t(v)$  of double phase energy~$\mathcal{J}(v)$ are nonlocal.  
\subsection{Subcritical case}
\label{sec:sub}
In this section we consider the subcritical case:~$\ind(W^{s,p}):=s-\frac{d}{p}<0$. The main result is formulated in the following theorem. 
\begin{theorem}\label{thm:sub}
 Let $1<p,q<\infty$,$0<s\le  t<1$  and ~$s-\frac d p <0$,~$\alpha>0$, and $tq > sp + \alpha$.  Then there
exists an integrand  with non negative, bounded  weight ~$a(\cdot,\cdot )\in C^{\alpha}\cap L^\infty(\setR^d\times \setR^d)$ such that Lavrentiev gap occurs. The choice~$t=1$ or~$s=1$ corresponds to the mixed non-local model result.  
\end{theorem}

To prove the Theorem~\ref{thm:sub} we first construct the competitor function~$u_C\in W\setminus H$. For the subcritical case it is defined as following.

\begin{definition}[\cite{BalDieSur20},~Competitor~$u_C$]
  \label{def:com}
  Let $\Omega := (-1,1)^d\times(-1,1)^d$ with~$d\geq 2$.  Let $$\frS := \frC^{d-1}_\lambda \times \set{0}$$ and $$\frD= \dim(\frS)=\frac{(d-1)\log 2}{\log(1/\lambda)},$$
    where~$\lambda \in (0,\frac 12)$.  Let $$\rho \in C^\infty(\setR^d\setminus \frS)$$
    be such that 
    \begin{enumerate}[label={(\roman{*})}]
    \item
      $\indicator_{\set{\dis(\overline{x},\frC^{d-1}_\lambda) \leq 2 \abs{x_d}}}
      \leq \rho \leq \indicator_{\set{\dis(\overline{x},\frC^{d-1}_\lambda)
          \leq 4 \abs{x_d}}}$.
    \item
      $\abs{\nabla \rho} \lesssim \abs{x_d}^{-1}
      \indicator_{\set{2 \abs{x_d} \leq \dis(\overline{x},\frC^{d-1}_\lambda) \leq
          4 \abs{x_d}}}$.
    \end{enumerate}
    We define
    \begin{align*}
      u_C(x) &:= \frac 12 \sgn(x_d)\, \rho(x).
      \end{align*}
\end{definition}
Later we use the competitor ~$u_C$ as a Dirichlet boundary value. For this we need to define function~$u_D$, which is a smooth version of a function~$u_C$. 
\begin{definition}\label{def: u_D}
 We define ${u}_D=\eta u_C$, with cut-off function $\eta \in C^\infty_0(\Omega)$,  $\indicator_{(-\frac46,\frac46)^d} \leq \eta \leq
  \indicator_{(-\frac56,\frac56)^d}$ and $\norm{\nabla \eta}_\infty \leq
  c$.
\end{definition}

In  \cite{BalDieSur20}[Proposition 14] the following estimate for the gradient of the competitor function~$u_C$ has been proved:  
\begin{lemma}[estimate for the competitor, local]
\begin{align}\label{subgrad}
  \abs{\nabla u_C} \lesssim \abs{x_d}^{-1} \indicator_{
                     \set{
                     2 \abs{x_d} \leq 
                     \dis(\overline{x},\frC_\lambda^{d-1}) \leq 4 \abs{x_d}
                     }
                     }.
\end{align}
\end{lemma}

The   weight~$a(y,z)$ for the example on Lavrentiev gap is defined as following 
\begin{definition}

 Let~$\rho^-, \rho^+ \in C^\infty(\setR^d\setminus \mathcal{\frS})$ be such that
\begin{align*}
\indicator _{\set{\dis(\abs{\overline{x}},\frC_\lambda^{d-1}) \le \frac{1}{2}x_d }} &\le \rho^-(x)\le \indicator _{\set{\dis(\abs{\overline{x}},\frC_\lambda^{d-1}) \le2x_d }},\quad x_d<0,\\
\indicator _{\set{\dis(\abs{\overline{x}},\frC_\lambda^{d-1}) \le \frac{1}{2}x_d }} &\le \rho^+(x)\le \indicator _{\set{\dis(\abs{\overline{x}},\frC_\lambda^{d-1}) \le 2x_d }},\quad x_d>0
\end{align*}
We define 
\begin{align*}
a(y,z):=\big(\rho^+(y)\rho^+(z)+\rho^-(y)\rho^-(z)\big)(\abs{y_d}^\alpha+\abs{z_d}^\alpha).
\end{align*}
\end{definition}

The weight~$a$ satisfies ~$a(x,y)\in C^\alpha(\overline{\Omega}\times\ \overline{\Omega})$ due to the construction.  Moreover,  we assume that~$a(x,y)$ vanishes outside~$3\Omega$.

\begin{lemma}[Behavior of the competitor, nonlocal]\label{lem:beh_com}
Let~$u_C$ be as in Definition~\ref{def:com}. For any~$y,z$ it holds

  \begin{align*}
    \abs{u_C(y)-u_C(z)}&\lesssim \indicator_{\set{\abs{y-z}\ge\frac 14 \abs{y_d}+\abs{z_d}}} (\indicator_{\mathcal{M}_4}(y)+\indicator_{\mathcal{M}_4}(z))\\&+\indicator_{\set{\abs{y-z}\le\frac 14 \abs{y_d}+\abs{z_d}}} \indicator_{\mathcal{M}_8}(y)\indicator_{\mathcal{M}_8 }(z)\frac{\abs{y-z}}{\abs{y_d}+\abs{z_d}}.
\end{align*}
\end{lemma}

\begin{proof}
 If~$y,z\in \mathcal{M}_4^\complement$, then by  construction~$u_C(y)=u_C(z)=0$ and the claim follows.    It remains to consider the case when~$y\in \mathcal{M}_4$ or~$z\in \mathcal{M}_4$. 
 
 By symmetry we can assume in the following that~$\abs{y_d}\ge\abs{z_d}$.  Hence we have $y\in\mathcal{M}_4$. If~$\abs{y-z}\ge \frac 14(\abs{y_d}+\abs{z_d})$ then the claim follows from~$\abs{u_C(y)-u_C(z)}\le 1$. Thus it suffices to consider the case~$\abs{y-z}< \frac 14(\abs{y_d}+\abs{z_d})$. In this case~$z\in \mathcal{M}_8$. Indeed we have 

\begin{align*}
    \dis(\overline z,\mathcal{C})\le \dis(\overline y, \mathcal C)+\abs{\overline y-\overline z}\le 4\abs{y_d}+ \frac 14 (\abs{y_d}+\abs{z_d})\\
    \le  \frac{17}{4}\abs{\overline z}+\frac 14 \abs{z_d}\le 8|z_d|. 
\end{align*}
Now let~$\abs{y-z}\le \frac 14(\abs{y_d}+\abs{z_d})$ and~$y,z\in\mathcal{M}_8$. It holds
    \begin{align*}
      \abs{u_C(y)-u_C(z)}&\le \biggabs{\int_0^1 \frac{d}{ds}u_C(y-s(y-z))\,ds}
      \\
                     &\leq \int_0^1\abs{\nabla u_C(z-s(y-z))}\abs{y-z}\, ds
      \\
                     &\leq \int_0^1\frac{c_1}{\abs{z-s(y-z)}}\,ds \, \abs{y-z}\leq \frac{2c_1\abs{y-z}}{\abs{y_d}+\abs{z_d}},
    \end{align*}
 where we used the estimate for~$\nabla u_C$ from~\eqref{subgrad}.   
 Now we combine all the cases to get the claim.
\end{proof}

The next step is to study how the energy~$\mathcal{J}^t_{q,a}$ behaves on smooth functions. In particular, we will show that for~$v\in C^0(\overline \Omega)$ with~$v=u_D$ on~$\Omega^\complement$ the~$t,q$-energy is positive $\mathcal{J}^t_{q,a}(v)\ge c_0>0$. Without  loss of generality we can assume that the function~$v$ is odd with respect to~$y_d$ and~$v(\overline y,0)=0$.

\begin{definition} We define the following sets 
 \begin{align*}
 K&:=\set{y\in \RRd: \abs{\overline{y}}\le \tfrac  14 y_d,\, 0<y_d\le 2}. 
\\
 K_j&:=\set{y\in K:  2^{-j}\le y_d\le 2^{-j+1}}. 
 \end{align*}

\end{definition}

\begin{lemma}\label{lem:riesz}
 For~$v\in \mathcal{C}^0([-2,2]^d)$  and~$x=(\overline x,0)$ there holds 
 \begin{align*}
  \abs{v(x)-\mean{v}_{x+K_0}}\le c\int_{x+K}\int_{x+K}  \frac{\abs{v(y)-v(z)}}{(\abs{y_d}+\abs{z_d})^{2d}} \indicator_{ \set{\abs{y-z}\le \abs{y_d}+\abs{z_d}}} \, dydz, 
 \end{align*}
where~$c=c(d)$.
\end{lemma}
\begin{proof}
  Using the continuity of~$v$ we have~$\lim_{j\to \infty}\mean{v}_{x+K_j}= v(x)$ and therefore
 \begin{align*}
  \abs{v(x)-\mean{v}_{x+K_0}}&=\lim_{j\to \infty } \abs{\mean{v}_{x+K_{j}}-\mean{v}_{x+K_0}}\\
  &\le\sum_{j\ge 0} \abs{\mean{v}_{x+K_{j+1}}-\mean{v}_{x+K_j}}\\ &\le \sum_{j\ge 0} \dashint_{x+K_{j+1}}\dashint_{x+K_{j}} \abs{v(y)-v(z)}\,dydz. 
 \end{align*}
Now if~$y\in x+K_{j+1}$ and~$z\in x+K_{j}$ then
\begin{align*}
    \abs{K_{j+1}}^{-1}\abs{K_j}^{-1}=2^{-2jd}\eqsim (\abs{y_d}+\abs{z_d})^{2d}.
\end{align*}
Therefore we obtain
 \begin{align*}
  \abs{v(x)-\mean{v}_{x+K_0}}&
  \lesssim\sum_{j\ge 0}\int_{x+K_j}\int_{x+K_j}\frac{\abs{v(y)-v(z)}}{(\abs{y_d}+\abs{z^d})^{2d}}dydz\\
  &\lesssim  \int_{x+K}\int_{x+K}  \frac{\abs{v(y)-v(z)}}{(\abs{y_d}+\abs{z_d})^{2d}} \, dydz.  
  \end{align*}
This proves the claim. 
\end{proof}

\begin{lemma}\label{lem:Lavr}
 For~$u_C$ defined in the definition~\ref{def:com} and~$\frD$ being the dimension of the fractal barrier set defined as $\frD := \dim(\frC_\lambda^m) = -m \log(2)/\log(\lambda)$ it holds:
  \begin{enumerate}
  \item  \label{i:a} If $s-\frac{d-\frD}{p}<0$, then  $\mathcal{J}^s_p(u_C)<\infty$. 
  \item \label{i:b} $\mathcal{J}^t_{q,a}(u_C)=0$.
  \item  \label{i:c} If $t-\frac{d-\frD}{q}+\frac{\alpha}{q}>0$, then   $\mathcal{J}^t_{q,a}(v)\ge c_0>0$ for all~$v\in C^0(\overline{\Omega})$ with~$v= u_D$ on~$ \Omega^{\complement}$.
 \end{enumerate}
 Properties~\ref{i:a} and~\ref{i:b}  ensures that so the competitor function~$u_C$ defined in \eqref{def:com} belongs to the energy space~$W_{u_D}$. 
\end{lemma}

\begin{proof}

First we use the properties of the competitor function from the Lemma~\ref{lem:beh_com}  to  split the ~$J^s_p(u_C)$ as following

\begin{align*}
  \mathcal{J}_p^s(u_C)&:=   \int_{{3\Omega}} \int_{{3\Omega}} \bigg(\frac{\abs{u_C(y)-u_C(z)}}{\abs{y-z}^s} \bigg)^p\frac{dy\, dz}{\abs{y-z}^d}\\
& \lesssim \int_{{3\Omega}} \int_{{3\Omega}} (\indicator_{\set{\mathcal{M}_4(y)+\mathcal{M}_4(z)}})\indicator_{\abs{y-z}\ge \frac 14 (\abs{y_d}+\abs{z_d})} \bigg(\frac{1}{\abs{y-z}^s} \bigg)^p\frac{dy\, dz}{\abs{y-z}^d}\\
&+ \int_{3\Omega} \int_{3\Omega} (\indicator_{\set{\mathcal{M}_8(y)\mathcal{M}_8(z)}})\indicator_{\set{\abs{y-z}\le \frac 14 (\abs{y_d}+\abs{z_d})}} \frac{\abs{y-z}^p}{(\abs{y_d}+\abs{z_d})^p}\frac{1}{\abs{y-z}^{sp+d}}\, dydz \\
&   :=\mathrm{I}+\mathrm{II}.
\end{align*}
We start with the estimate for~$\mathrm{I}$:
\begin{align*}
 \mathrm{I}= \int_{3\Omega}\int_{3\Omega} (\indicator_{\set{\mathcal{M}_4(y)+\mathcal{M}_4(z)}})\indicator_{\set{\abs{y-z}\ge \frac 14 (\abs{y_d}+\abs{z_d})}} \abs{y-z}^\gamma \,dydz,\quad \gamma=-ps-d.
\end{align*}
Now, using symmetry and Lemma~\ref{lem:cantor-estimates}, we obtain  
\begin{align*}
    \mathrm{I}\lesssim  2  \int_{3\Omega}  \indicator_{\set{\mathcal{M}_4(z)}} \abs{z_d}^{-\gamma+d} dz <\infty,
\end{align*}
provided~$$-\gamma>-2d+\frD,\quad -ps-d>-2d+\frD,$$ which is equivalent to~$$s-\frac{d-\frD}{p}<0.$$ 

We continue with estimating the second term~$\mathrm{II}$. By symmetry we get
\begin{align*}
   \mathrm{II}= &\int_{3\Omega}\int_{3\Omega}\indicator_{\set{\mathcal{M}_8(y)}}\indicator_{\set{\mathcal{M}_8(z)}}\indicator_{\set{\abs{y-z}\le \abs{y_d}+\abs{z_d}}}\abs{y-z}^{p-sp-d}(\abs{y_d}+\abs{z_d})^{-p}\,dydz\\
    &\lesssim \int_{3\Omega} \indicator_{\set{\abs{y_d}\le \abs{z_d}}}\indicator_{\set{\mathcal{M}_8(y)}}\abs{z_d}^{-sp}dz<\infty,
\end{align*}
provided~$-sp-d>-2d+\frD$, so~$s-\frac{d-\frD}{p}<0$.

Next, we have  $\mathcal{J}^t_{q,a}(u_C)=0$ due to the construction of the competitor function.

Now we prove the property (c). For this let us consider a function~$v(x)\in C_0(\overline{{3\Omega}})$.    By symmetrization we can assume that~$v$ is odd in~$x_d$, so~$v(\overline{x}, 0)=0$. Indeed, let 
\begin{align*}
 \hat v(\overline x, x_d):=\frac{v(\hat x, x_d)-v(\hat x,-x_d)}{2}.
\end{align*}
Then~$\hat v$ is odd in~$x_d$ and~$\mathcal{J}(\hat v)\le \mathcal{J}(v)$ by convexity and symmetry of~$\mathcal{J}$. Hence it suffices to prove the claim for~$v$ odd in~$x$. In particular, we have~$v(\hat x, 0)=0$.
Now we use Lemma~\ref{lem:riesz}
  \begin{align*}
   0&<\frac 12 \int_{\frC} d\mu(x)= \int_{\frC} \abs{v(\overline{x},0)-\mean{v}}d\mu(x) \le \int_{\frC}   \int_{(\overline{x},0)+K}\int_{(\overline{x},0)+K} \frac{\abs{v(y)-v(z)}}{(\abs{y_d}+\abs{z_d})^{2d}}d\overline{x} dydz\\
   &\le \int_{\mathcal{M}_4^+}\int_{\mathcal{M}_4^+} \frac{\abs{v(y)-v(z)}}{(\abs{y_d}+\abs{z_d})^{2d}} dydz \int_{\frC}  \indicator_{\set{y\in(\overline x,0)+K,z\in(\overline x,0)+K }}d\overline x:=\mathrm{III},
   \end{align*}
and by the definition of the set~$K$ we obtain

\begin{align*}
 \int_{\frC}  \indicator_{\set{y\in(\overline x,0)+K,z\in(\overline x,0)+K }}d\overline x\lesssim \indicator_{\set{\abs{\overline y-\overline z}\le\frac 14 (\abs{y_d}+\abs{z_d})}} (\abs{y_d}+\abs{z_d})^\frD,
\end{align*}
and therefore
\begin{align*}
\mathrm{III}\lesssim   \int_{\mathcal{M}_4^+}\int_{\mathcal{M}_4^+} \frac{\abs{v(y)-v(z)}}{(\abs{y_d}+\abs{z_d})^{2d-\frD}}\indicator_{\set{\abs{\overline y-\overline z}\le\frac 14 (\abs{y_d}+\abs{z_d})}} dydz.
\end{align*}
With Hölder's inequality we estimate
\begin{align*}
    \mathrm{III}&\le \mathrm{I}\cdot \mathrm{II},
\end{align*}
where
\begin{align*}
   \mathrm{I}&:=  \bigg(\int_{\mathcal{M}^+_4} \int_{\mathcal{M}^+_4}  (\abs{y_d}+\abs{z_d})^\alpha\bigg(\frac{\abs{v(y)-v(z)}}{\abs{y-z}^t}\bigg)^q\frac 1{\abs{y-z}^d}dydz\bigg)^\frac 1 q ,\\  
 \mathrm{II}&:=  \bigg(\int_{\mathcal{M}^+_4}\int_{\mathcal{M}^+_4} (\abs{y_d}+\abs{z_d})^{\alpha(q'-1)+(-2d+\frD)q' }\abs{y-z}^{tq'+dq'-d}\indicator_{\set{\abs{y-z}\le   \frac14(\abs{y_d}+\abs{z_d})}} dy dz \bigg)^\frac 1 {q'}\\
 &=\bigg(\int_{\mathcal{M}^+_4}\int_{\mathcal{M}^+_4} (\abs{y_d}+\abs{z_d})^{\gamma_1}\abs{y-z}^{\gamma_2}\indicator_{\set{\abs{y-z}\le \frac14(\abs{y_d}+\abs{z_d})}} dy dz \bigg)^\frac 1 {q'}
\end{align*}
with~$\gamma_1=\alpha(q'-1)-2dq'+\frD q'$,~$\gamma_2=tq'+dq'-d>0$.  Now by symmetry
\begin{align*}
    &\int_{\mathcal{M}^+_4}\int_{\mathcal{M}^+_4} (\abs{y_d}+\abs{z_d})^{\gamma_1}\abs{y-z}^{\gamma_2}\indicator_{\set{\abs{y-z}\le \frac14(\abs{y_d}+\abs{z_d})}} dy dz \\
   &\lesssim \int_{\mathcal{M}^+_4}\int_{\mathcal{M}^+_4} \abs{z_d}^{\gamma_1}\indicator_{\set{\abs{y_d}\le \abs{z_d}}}\abs{y-z}^{\gamma_2}\indicator_{\set{\abs{y-z}\le \frac14(\abs{y_d}+\abs{z_d})}} dy dz \\
    &\lesssim \int_{\mathcal{M}^+_4} \abs{z_d}^{\gamma_1+d}\abs{z_d}^{\gamma_2} dz<\infty 
\end{align*}
provided~$\gamma_1+\gamma_2>-2d+\frD$ due to Corollary~\ref{cor:cor}. So $$\alpha(q'-1)-2dq'+\frD q'+tq'+dq'-d>-2d+\frD$$ and therefore~$$t-\frac{d-\frD}{q}+\frac{\alpha}{q}>0,$$ which completes the proof.
\end{proof}

With the help of the Lemma~\ref{lem:Lavr} we are ready now to complete the proof of the main result of this section.  

\begin{proof}[Proof of Theorem \ref{thm:sub}]

We have:
\begin{enumerate}
  \item If $s-\frac{d-\frD}{p}<0$, then  $\mathcal{J}^s_p(u_C)<\infty$, so the competitor function~$u$ defined in \eqref{def:com} belongs to the energy space~$W_{u_D}$.
  \item $\mathcal{J}^t_{q,a}(u_C)=0$.
  \item  If $t-\frac{d-\frD}{q}+\frac{\alpha}{q}>0$, then   $\mathcal{J}^t_{q,a}(v)\ge c_0>0$ for all~$v\in C^0(\overline{\Omega})$  with~$v= u_D$ on~$ \Omega^{\complement}$.
 \end{enumerate}

 Now, we use the argument with the scaling of the weight coefficient~$a(x,y)$. Let~$\nu>0$ be a free parameter  and define the energy with the scaled weight~$\nu a$ 
 \begin{align*}
    \mathcal{J}_\nu(v):= \mathcal{J}^s_p(v)+\mathcal{J}^t_{q,\nu a}.
 \end{align*}
Then
 \begin{align*}
    \mathcal{J}(v)\ge \mathcal{J}^t_{q,\nu a}(v)=\nu \mathcal{J}^t_{q,a}(v)\ge \nu c_0>0,
 \end{align*}
since~$v=u_D$ on~$ \Omega^{\complement}$. From the other side choose~$\nu$ such that
 \begin{align*}
    \mathcal{J}^s_p(u_C)\le \nu c_0-1.
 \end{align*}
Then
 
 \begin{align*}
\mathcal{J}^s_p(u_C)\le \nu c_0-1\le \mathcal{J}^s_p(v)-1,
 \end{align*}
so there Lavrentiev phenomenon occurs for sufficiently large value of~$\nu$. 

Combination of properties (a) and (c) gives the sufficient condition on the parameters of the model:

\begin{align*}
    s-\frac{d-\frD}{p}< t-\frac{d-\frD}{q}+\frac{\alpha}{q},
\end{align*}
and since~$\frD\in (0,1)$ could be chosen arbitrary, that is equivalent to 
\begin{align*}
 tq>sp+\alpha.
\end{align*}
\end{proof}

We now  provide a straightforward embedding interpretation of the condition 

$$tq>sp+\alpha$$
we have to determine the existence of a Lavrentiev gap. Let~$\mathcal{C}=\mathcal{C}_\lambda$ be the Cantor necklace with the Hausdorff dimension~$$\frD:=\frac{\log(2)}{ \log(\lambda^{-1})}.$$ For the subcritical case it holds
\begin{align*}
 H^{s,p}(\Omega) \not\embedding L^p(\mathcal{C}),
\end{align*}
that is
\begin{align}\label{eq:1}
 \ind(H^{s,p}(\Omega))=s-\frac p d < \ind( L^p(\mathcal{C}))=-\frac{\frD}{p}.
\end{align}

From the other side
\begin{align*}
 H^{t,q}_a(\Omega) \embedding L^q(\mathcal{C}),
\end{align*}
that is
\begin{align}\label{eq:2}
 \ind(H^{t,q}_a(\Omega))=t-\frac q d-\frac \alpha d \ge \ind( L^q(\mathcal{C}))=-\frac{\frD}{q}.
\end{align}
The combination of~\eqref{eq:1} and~\eqref{eq:2} and the choice of~$\frD$ give the condition~$tq>sp+\alpha$. 

\begin{remark}
The critical case~$s-\frac d p=0$ is treated in the similar way. The geometrical configuration in this case corresponds to one saddle point set up (see \textcite{Zhi95}; \textcite{EspLeoMin04} for the local case). The Lavrentiev gap condition depends on the space dimension in this case and takes the form~$sp+\alpha<d<tq$.
\end{remark}

\begin{remark}
 It has been shown by \textcite{BOS22} that, if~$1<p\le q<\infty$,~$0<s\le t<1$ and~$tq\le sp+\alpha$, then every weak solution of the nonlocal double phase problem~\textbf{IV} in the subcritical case~$s-\frac d p<0$ is locally Hölder continuous. Thus, we can compare our condition 
for the Lavrentiev gap with  for the absence of the Lavrentiev gap. This shows that our condition~$tq>sp+\alpha$ is sharp in the subcritical case.

\end{remark}

\section{Supercritical case}\label{sec:super}

In this section we consider the supercritical case~~$\ind(W^{s,p})=s-\frac{d}{p}>0$. The main result is formulated in the following theorem. 
\begin{theorem}\label{thm:super}
 Let $p,q>1$,~$s-\frac d p >0$,~$\alpha\in [0,1]$, and 
 \begin{align*}
\frac{p}{p-1}\bigg(s-\frac d p\bigg)< \frac{q}{q-1}\bigg(t-\frac {d+\alpha}{q}\bigg).
 \end{align*}
Then there exists an integrand  with non negative, bounded  weight ~$a\in C^{\alpha}(\overline{\Omega}\times \overline{\Omega})$ such that Lavrentiev gap occurs.  The choice~$t=1$ or~$s=1$ corresponds to the mixed non-local model result.  
\end{theorem}

To prove the Theorem~\ref{thm:super} we first construct the competitor function~$u\in W\setminus H$. For the supercritical case it is defined as following.

\begin{definition}[\cite{BalDieSur20},\,  Competitor~$u_C$]
  \label{def:com_super}
  Let $\Omega := (-1,1)^d\times(-1,1)^d$ with~$d\geq 2$.  Let $$\frS := \frC^{d-1}_\lambda \times \set{0}$$ and $$\frD= \dim(\frS)=\frac{\log 2}{\log(1/\lambda)},$$ where~$\lambda \in (0,\frac 12)$ and~$\mu_\lambda$ be the Cantor measure defined in the Section~\ref{sec:Cantor}.
    We define the competitor function 
    \begin{align*}
      u_C(x) &:= (\delta_0^{d-1}\times \mu_\lambda) *u_d. 
      \end{align*}

\end{definition}

In the Proposition 14, \cite{BalDieSur20} the following estimate for the gradient of the competitor function~$u$ was obtained:  
\begin{align}\label{supergrad}
  \abs{\nabla u_C} \lesssim \abs{\overline x}^{\frD-1} \indicator_{
                     \set{\dis({x}_d,\frC_\lambda) \leq \frac 12  \abs{\overline x}
                     }
                     }.
\end{align}

We also need to work with separate elements of the Cantor necklace~$\mathcal{N}$. 
For each removed $j$-th interval of the generation $l$ we denote the value of $v$ in its upper point by $v_{l,j,+}$ and the value of $v$ in its lower point by $v_{l,j,-}$, and the corresponding element of the Cantor necklace by $\mathcal{N}_{l,j}$. Note, that~$\mathcal{M}^\complement=\mathcal{N}$.

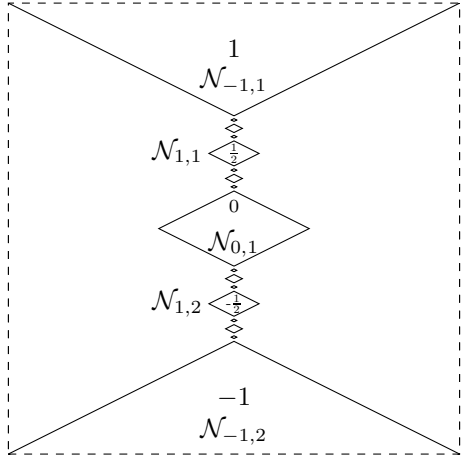
\begin{figure}[ht!]\centering
  
  \begin{tikzpicture}[scale=3]
    \draw[dashed] (-1,-1) -- (-1,+1) -- (+1,+1) -- (+1,-1) -- cycle;

    \clip (-1,-1) rectangle (1,1);
    
    
    \vfullpattern{
      \filldraw[white] (0,-1/6) -- (1/3,0) -- (0,1/6) -- (-1/3,0) -- (0,-1/6); 
      \draw  (0,-1/6) -- (1/3,0) -- (0,1/6) -- (-1/3,0) -- (0,-1/6); 
    }

    \node at (0,-3/4) {$-1$};
   
    \node at (0,-0.90) {$\mathcal{N}_{-1,2}$};
   
   \node at (0,0.80) {$1$};
    \node at (0,0.65) {$\mathcal{N}_{-1,1}$};

    \node at (0,0.1) {$\scriptstyle 0$};

    \node at (0,-0.07)  {$\small{\mathcal{N}_{0,1}}$};
    
    \node at (0,-1/3) {\scalebox{0.6}{$\text{-}\frac 12$}};
    \node at (-1/4 ,-1/3) {$\small{\mathcal{N}_{1,2}}$};
    
    \node at (0,+1/3) {\scalebox{0.6}{$\frac 12$}};

    \node at (-1/4 ,1/3) {$\small{\mathcal{N}_{1,1}}$};
    
    \draw[dashed] (-1,-1) -- (-1,+1) -- (+1,+1) -- (+1,-1) -- cycle;1
       
  \end{tikzpicture} 
  \caption{The competitor~$u$, supercritical case}
\end{figure}

Now we define the transition  weight~$a(x,y)$ for the example on Lavrentiev gap. Due to~\cite{BalDieSur20} there exists~$\rho \in C^\infty(\setR^d\setminus \frC)$ such as 
    \begin{enumerate}
    \item
      $\indicator_{\set{d(x_d,\frC_\lambda)(x) \leq \frac 1 2 \abs{\overline{x}}}}
      \leq \rho_a (x) \leq \indicator_{\set{d(x_d,\frC_\lambda)
          \leq 2 \abs{\overline{x}}}}(x)$.
    \item
      $\abs{\nabla \rho_a(x)} \lesssim \abs{\overline{x}}^{-1}
      \indicator_{\set{\frac 12 \abs{\overline{x}} \leq
          d(x_d,\frC_\lambda) \leq 2 \abs{\overline{x}}}}
      $
    \end{enumerate}
    and
 \begin{align*}
  \rho_j(x)=\indicator_{\mathcal{N}_{j,l}}\rho_a(x ).
 \end{align*}

We define the weight~$a(x,y)$ such that
    \begin{align*}
      a(x,y) := \sum_j\abs{\overline{x}}^\alpha(1-\rho_j(x)\rho_j(y)) +\abs{\overline{y}}^\alpha(1-\rho_j(x)\rho_j(y)).
    \end{align*}

The weight~$a$ satisfies ~$a(x,y)\in C^\alpha(\overline \Omega\times \overline \Omega)$ due to the construction and is assumed to vanish outside of~$3\Omega$.

  \begin{lemma} [Behavior of the competitor]\label{lem:beh_com_super}
  For the competitor function~$u_C$ from  the Definition~\ref{def:com_super} there holds
  \begin{align*}
      \abs{u_C(y)-u_C(z)}\lesssim
      \min \biggset{1, \frac 1 {\frD}\abs{y-z}^{\frD}, \frac{\abs{y-z}}{(\abs{\overline{y}}\vee \abs{\overline{z}})^{1-\frD}}}.
    \end{align*}
  \end{lemma}
  \begin{proof}

    We first consider the situation when~$y\in \mathcal{N}$ and~$\abs{y-z}\le \frac 12$. In this case in fact the segment~$$x:=[y,z]_t=(1-t)y+tz$$ belongs to~$\mathcal{N}$. Indeed let~$\abs{\overline{z}}\le \abs{\overline y}$. Then
    \begin{align*}
      \abs{\overline x}\le (1-t)\abs{\overline y}+t\abs{\overline z}\le \abs{\overline y},\\
      \abs{\overline x}\ge \abs{\overline y}-t\abs{\overline y-\overline z}\ge \tfrac 12 \abs{\overline y}
    \end{align*}
    and hence
    \begin{align*}
      \tfrac 12 \abs{\overline y}\le \abs{\overline x}\le \abs{\overline y}.
    \end{align*}
    Now 
    \begin{align*}
      \abs{\overline y}\le \tfrac 14 \dis(y_d,\frC_m^\lambda)\le \tfrac 14\dis(x_d,\frC_m^\lambda)+\tfrac 14\abs{y_d-x_d}
    \end{align*}
    and we see that the segment~$x$ belongs to~$\mathcal{N}$.
    
  \textbf{Case 1:} $\abs{y-z}\le \frac 12(\abs{\overline y}\vee\abs{\overline z})$. In this case, assuming~$\abs{\overline z}\le \abs{\overline y}$, we also have
    \begin{align*}
      \abs{\overline y}&\le \abs{\overline z}+\abs{y-z}\le\abs{\overline z}+\tfrac 12\abs{\overline y}\le 2\abs{\overline{z}}.
    \end{align*}
    
    Let~$x=z+t(y-z)$ with~$0<t<1$. No, using~$\abs{\overline z}\le \abs{\overline y}$ we have
    \begin{align*}
      \abs{z-x}&=(1-t)\abs{y-z}\le \tfrac{1}{2}\abs{\overline{y}},
      \\
      \abs{\overline{x}}&\le \abs{y}\le 2\abs{\overline x}.
    \end{align*}
    By the estimate \eqref{supergrad} we get
    \begin{align*}
      \abs{u_C(y)-u_C(z)}&\le \int_0^1\abs{\nabla u (x(t))}\, dt\abs{y-z}\\
      &\le  c\int_0^1\abs{\overline x}^{\frD-1}\, dt\abs{y-z}\\      
      &\le c \frac{\abs{y-z}}{(\abs{\overline y}\vee\abs{\overline z}))^{1-\frD}},
    \end{align*}
    where we used that~$\abs{\overline x}\ge \frac 12 \abs{\overline y}$ at the last step.
    
    \textbf{Case 2: }~$\abs{y-z}\ge \frac 12(\abs{\overline y}\vee\abs{\overline z})$. We consider
    \begin{align*}
      \abs{u_C(y)-u_C(0,y_d)}=\abs{u_C(y)-u_C(\hat{y})}
    \end{align*}
    with~$\hat{y}=(\frac 12 \dis(y_d,\frC^m_\lambda),y_d)$. Now
    \begin{align*}
      \abs{u_C(y)-u_C(0,y_d)}\le \int_0^1\abs{\nabla u_C((0,y_d)+t(\overline{y},0))}\, dt\abs{\overline{y}}\lesssim \frac  1{\frD}\abs{\overline{y}}^{\frD}.
    \end{align*}
    The same estimate with~$z$ follows by symmetry.
  \end{proof}
  
  Now we  proof that the competitor~$u_C$ belongs to the energy space~$W_{u_D}$. First, we consider the most complicated case, when both points~$y$ and~$z$ belong to the different diamonds~$\mathcal{N}_{l,j}$ and~$\mathcal{N}_{k,i}$ of the Cantor necklace~$\mathcal{N}$. This is in particular equivalent to the fact, that the components~$y_d$ and~$z_d$ on the line corresponding to the the~$d$ coordinate are in~$[-1,1]\setminus \mathcal{C}_\lambda^{d-1}$.    The next auxiliary result allows to integrate the distance function from the point to the Cantor set. 
  \begin{lemma}
   \label{lem:super_dist}
   If~$\beta<-1+\frD$, then
   \begin{align*}
        \int_{[-1,1]\setminus\mathcal{C}_\lambda^m} \dis(y_d, \mathcal{C}_\lambda^m)^{-\beta} dy_d <\infty.
   \end{align*}
  \end{lemma}
  \begin{proof}
    
  We estimate  the integral with respect to the full Cantor necklace~$\mathcal{N}$ via the sum taken over all the diamonds~$\mathcal{N}_{l,j}\in \mathcal{N}$ and over the level sets of the Cantor necklace~$l\in\mathcal{L}$. From here~$\delta_l$ denotes the diameter  of the element~$\mathcal{N}_{l,j}$.  We have

  \begin{align*}
  \int_{[-1,1]\setminus\mathcal{C}_\lambda^m} \dis(y_d, \mathcal{C}_\lambda^m)^{-\beta} dy_d &=\sum_j\int_{\mathcal{N}_j}\dis(y_d,\frC_\lambda^m)^{-\beta}dy_d\\
  &\lesssim \sum_j \frac{\delta_j^{1-\beta}}{1-\beta}=\sum_{l\in\mathcal{L}} \# l\frac{\delta_j^{1-\beta}}{1-\beta}\lesssim \sum_{l\in\mathcal{L}} \lambda^{-\frD l}\frac{\lambda^{l(1-\beta)}}{1-\beta}\\&\lesssim \sum_{l\in\mathcal{L}}\lambda^{(-\frD+1-\beta)l}<\infty
 \end{align*}
provided~$\beta>-1+\frD$. This proves the claim.
  \end{proof}
  \begin{lemma}
  \label{lem:com_super}
  
  If $p<\frac{d-\frD}{s-\frD}$,~$0<\frD<s<1$, then  $\mathcal{J}^s_p(u_C)<\infty$, so the competitor function~$u$ defined in the Definition~\ref{def:com_super} belongs to the energy space~$W_{u_D}$.
    \end{lemma}
\begin{proof}
Depending on the location of the points~$y$ and~$z$, we distinguish the following cases:
  \begin{enumerate}
  \item Let~$y$ and~$z$ both belong to the region outside  the Cantor Necklace~$\mathcal{N}$:  ~$y,z \in\mathcal{M}$.  In this case, using Lemma~\ref{lem:beh_com_super}  we get
  \begin{align*}
    &\int_\mathcal{M} \int_\mathcal{M}  \frac{\abs{u_C(y)-u_C(z)}^p}{\abs{y-z}^{sp+d}} dydz\\
    &\lesssim \int_\mathcal{M} \int_\mathcal{M}\indicator_\set{\abs{y-z}\ge\frac 12  (\abs{\overline y}+\abs{\overline z})}\abs{y-z}^{p\frD-sp-d} dzdy   \\&+ \int_\mathcal{M} \int_\mathcal{M}\indicator_\set{\abs{y-z}<\frac 12  (\abs{\overline y}+\abs{\overline z})} \abs{y-z}^{p-sp-d} (\abs{\overline y}+\abs{\overline z})^{(\frD-1)p} dydz <\infty,
  \end{align*}
  provided~
  \begin{align*}
  p\frD-sp-d&>-2d+\frD,\\ 
  p-sp-d&>-d,\\
  p-sp-d+(\frD-1)p&>-2d+\frD.
  \end{align*}
That leads to the condition~$p<\frac{d-\frD}{s-\frD}$.
  \item Let  ~$y\in \Omega$  and~$z\in\mathcal{M}$. In this case we have~$\abs{\overline y}+\abs{\overline z}\lesssim \abs{y-z}$, using Lemma~\ref{lem:beh_com_super} we get
    \begin{align*}
    &\int_{3\Omega} \int_\mathcal{M}  \frac{\abs{u_C(y)-u_C(z)}^p}{\abs{y-z}^{sp+d}} dydz\\
    &\lesssim \int_{\Omega}  \int_\mathcal{M}\indicator_\set{\abs{y-z}\ge\frac 12  (\abs{\overline y}+\abs{\overline z})} \abs{y-z}^{p\frD-sp-d} dzdy \lesssim \int_{\mathcal{M}} \abs{\overline z}^{p\frD-sp}dz <\infty.
  \end{align*}
   Here  we used Corollary \ref{cor:cor} at the last step. The last integral is finite provided $$p\frD-sp-d>-2d+\frD.$$  
\item Let ~$y\in \Omega$  and~$z\in\mathcal{M}^\complement=\mathcal{N}$. In this case by the construction of the competitor function~$u_C$ we have~$u_C(y)-u_C(z)=0$.
  
   \item Let~$y$ and~$z$ belong to different diamonds of the Cantor Necklace~$\mathcal{N}$ : $y \in \mathcal{N}_{l,j}, z\in \mathcal{N}_{k,i}$ with~$k\neq l$ or~$j\neq i$.  In this case~$$\dis(y_d, \frC)+\dis(z_d, \frC)\lesssim \abs{y_d-z_d},$$
   and if~$p<\frac{d-\frD}{s-\frD}$, then
   $$
   \mathcal{I}:=\int_{3\Omega} \int_{3\Omega}  \frac{\abs{u_C(y)-u_C(z)}^p}{\abs{y-z}^{sp+d}} dydz<\infty.
   $$
   Indeed, using Lemma~\ref{lem:beh_com_super} and  evaluating the integral with respect to~$\overline y$ and~$\overline z$  we get 
   \begin{align*}
    &\int_{3 \Omega} \int_{3 \Omega} \sum_{l,j,k,i} \frac{\abs{u_C(y)-u_C(z)}^p}{\abs{y-z}^{sp+d}} \indicator_{\mathcal{N}_{l,j}}\indicator_{\mathcal{N}_{k,i}} \indicator_{\set{\dis(y_d, \frC)+\dis(z_d, \frC)\lesssim \abs{y_d-z_d}}}   dydz  \\
    &\lesssim \int_{-1}^{1} \int_{-1}^{1}  \abs{y-z}^{\frD p-sp-d}\dis(y_d, \frC)^{d-1}\dis(z_d, \frC)^{d-1} \indicator_{\set{\dis(y_d, \frC)+\dis(z_d, \frC)\lesssim \abs{y_d-z_d}}}  dy_ddz_d \\
    &\lesssim 2\int_{-1}^{1}  \int_{-1}^1 \abs{y-z}^{\frD p-sp-d} \dis(y_d, \frC)^{2(d-1)}  \indicator_{\set{\dis(y_d, \frC)+\dis(z_d, \frC)\lesssim \abs{y_d-z_d}}} dy_ddz_d\\
    &\lesssim  \int_{[-1,1]\setminus\mathcal{C}_\lambda^m} \dis(y_d, \frC)^{-(\frD p-sp-d+1+2(d-1))} dy_d, 
   \end{align*}by Lemma~\ref{lem:super_dist} the last integral is finite provided~$$\frD p-sp-d+1+2(d-1)<-1+\frD,$$ so~$p<\frac{d-\frD}{s-\frD}$.
  \end{enumerate}

\end{proof}

We denote by~$x^+_{l,j}$ and $x^-_{l,j}$ the upper and lower  vertexes of an element of the necklace~$\mathcal{N}$, see the Figure 2.

The next step is to study how the energy~$\mathcal{J}^t_{q,a}$ behaves on  smooth functions. In particular, we will show that for~$v\in C^0(\overline \Omega)$, there holds $\mathcal{J}^t_{q,a}(v)\ge c_0>0$.

\begin{lemma}\label{lem:superRiesz}
 For every~$v\in C^0(\overline{\Omega})$ it holds
 \begin{align}
  \abs{v(x^+_{l,j})-v(x^-_{l,j})}\lesssim \int_{\mathcal{N}_{l,j}} \int_{\mathcal{N}_{l,j}}\frac{\abs{v(y)-v(z)}}{(\abs{\overline y}+\abs{\overline z})^{2d}} \indicator_{\mathcal{M}_8(y)}\indicator_{\mathcal{M}_8(z)}dy dz.
 \end{align}

\end{lemma}
\begin{proof}
    There exists a sequence of balls~$B_k$,~$k\in \setR$, $B_k\subset 2B_{k+1}$~ with the following properties:
    \begin{enumerate}
    \item $B_k\to\set{x_{l,j}^+}$ for~$k\to \infty$;
     \item $B_k\subset \mathcal{N}_{j,k} \cap \set{\abs{\overline y}>\frac {1}{10}d(y_d,\frC)} \cap \set{\abs{\overline z}>\frac {1}{10}d(z_d,\frC)} $;
     \item $\text{diam} B_k\eqsim2^{-k} \abs{x^+_{l,j}-x^-_{l,j}} $.
    \end{enumerate}
Note also, that if~$y\in B_k$, then~$\abs{\overline y}\eqsim d(y_d, \frC)$.    
Now, since the function~$v$ is continuous, for~$y\in B_k, z\in B_{k-1}$ we have
\begin{align*}
 \abs{v(x^+_{l,j})-v(x^-_{l,j})}&\le  \sum_{k\in \setR} \abs{\mean{v}_{B_k}-\mean{v}_{B_{k-1}}}\\
 &\lesssim \sum_m \dashint_{2B_k}\dashint_{2B_{k}} \abs{v(y)-v(z)} dydz\\
 &\lesssim\sum_k\int_{2B_k}\int_{2B_{k}} \frac{\abs{v(y)-v(z)}}{(\abs{\overline y}+\abs{\overline z})^{2d}}\, dy dz \\
 &\lesssim \int_{N_{j,k}}\int_{N_{j,k}} \frac{\abs{v(y)-v(z)}}{(\abs{\overline y}+\abs{\overline z})^{2d}} \indicator_{\mathcal{M}_8(y)}\indicator_{\mathcal{M}_8(z)} dydz,
\end{align*}
where we used at the last step the fact that the balls in~$B_k$ have finite overlap.

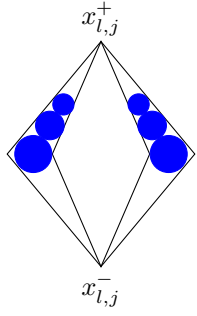
\begin{figure}[ht!]\centering \label{pic1}
 \begin{tikzpicture}

     \node [draw,diamond,
     draw = black,
    minimum width = 1.3cm,
    minimum height = 3cm
     ]{};
     \node at (0 ,1.8) {$x^+_{l,j}$};
     \node at (0 ,-1.8) {$x^-_{l,j}$};
    
     \node[diamond,
    draw = black,
    minimum width = 2.5cm,
    minimum height = 3cm]{} ; 
   \draw[fill,color=blue] (-0.9,0) circle[radius=7pt];
    \draw[fill,color=blue] (-0.68,0.38) circle[radius=5.4pt];
   \draw[fill,color=blue] (-0.5,0.66) circle[radius=4pt];
    \draw[fill,color=blue] (0.9,0) circle[radius=7pt];
        \draw[fill, color=blue] (0.68,0.38) circle[radius=5.4pt];
   \draw[fill,color=blue] (0.5,0.66) circle[radius=4pt];
    \end{tikzpicture}
    \caption{One element of the necklace~$\mathcal{N}_{l,j}$ and the sequence of covering balls~$B_k$.}
\end{figure}
\end{proof}

The auxiliary results of Lemma~\ref{lem:superRiesz} help to prove  the following key  properties of the model in the supercritical case

\begin{lemma}\label{lem:Lavr_super}
 For~$u_D$ defined in the Definition~\ref{def:com} and~$\frD$ being the dimension of the fractal barrier set defined as $\frD := \dim(\frC_\lambda^m) = -m \log(2)/\log(\lambda)$ it holds:
  \begin{enumerate}
  \item If $p<\frac{d-\frD}{s-\frD}$, then  $\mathcal{J}^s_p(u_C)<\infty$.
  \item $\mathcal{J}^t_{q,a}(u)=0$.
  \item  If 
  \begin{align}\label{eq:cond1}
 \frac{q}{q-1}\bigg(t-\frac {d+\alpha}{q}\bigg)>\frD,
  \end{align}then   $\mathcal{J}^t_{q,a}(v)\ge c_0>0$ for~$v\in C^{0,\gamma}(\overline{\Omega})$,~$\gamma>\frD$, with~$v= u_D$ on~$ \Omega^{\complement}$.
 \end{enumerate}
Properties (a) and (b) implies, that  the competitor function~$u_C$ defined in \eqref{def:com_super} belongs to the energy space~$W_{u_D}$.
 \end{lemma}

\begin{proof}
 The property (a) was proved in Lemma~\ref{lem:com_super}.
Now, $\mathcal{J}^t_{q,a}(u)=0$ due to the construction of the competitor function.

To show (c) we use the following property for~$v\in C^0(\overline{\Omega})$ with~$v=u_D$ on~$\Omega^C$:
 \begin{align*}
  1\lesssim \sum_l\sum_{j=1}^{2^l-1}\abs{v(x_{l,j}^+)-v(x^-_{l,j})}.
 \end{align*}

No we use Lemma~\ref{lem:superRiesz} and  Hölder inequality  
\begin{align*}
 &\int_{\mathcal{N}_{l,j}} \int_{\mathcal{N}_{l,j}}\frac{\abs{v(y)-v(z)}}{(\abs{\overline y}+\abs{\overline z})^{2d}} \indicator_{\mathcal{M}_8(y)}\indicator_{\mathcal{M}_8(z)}dy dz\\
 &\lesssim \bigg(\int_{\mathcal{N}_{l,j}} \int_{\mathcal{N}_{l,j}}a(y,z)\frac{\abs{v(y)-v(z)}^q}{\abs{y-z}^{tq+d}}dydz \bigg)^{\frac 1 q} \bigg(\int_{\mathcal{N}_{l,j}} \int_{\mathcal{N}_{l,j}}    \abs{y-z}^{\gamma_1} (\abs{\overline y}+\abs{\overline z})^{\gamma_2}\indicator_{\mathcal{M}_8(y)}\indicator_{\mathcal{M}_8(z)} dydz\bigg)^{\frac 1 {q'}}, 
\end{align*}
with~$\gamma_1=tq'+\frac{d}{q-1}$ and~$\gamma_2=-2dq'-{\alpha}\frac{q'}{q}$ and the last integral is finite provided 
$$
\gamma_1 +\gamma_2 >-2d+\frD,
$$
which gives the condition~\eqref{eq:cond1}.
\end{proof}

With the help of the Lemma~\ref{lem:Lavr_super} we are ready now to complete the proof of the main result of this section.  
\begin{proof}[Proof of Theorem \ref{thm:super}]
 Now, we use the argument with the scaling of the weight coefficient~$a(x,y)$. Let~$\nu>0$ be a free parameter  and define the energy with the scaled weight~$\nu a$ 
 \begin{align*}
    \mathcal{J}_\nu(v):= \mathcal{J}^s_p(v)+\mathcal{J}^t_{q,\nu a}.
 \end{align*}
Then
 \begin{align*}
    \mathcal{J}(v)\ge \mathcal{J}^t_{q,\nu a}(v)=\nu \mathcal{J}^t_{q,a}(v)\ge \nu c_0>0,
 \end{align*}
since~$v=u_D$ on~$ \Omega^{\complement}$. From the other side choose~$\nu$ such that
 \begin{align*}
    \mathcal{J}^s_p(u_C)\le \nu c_0-1.
 \end{align*}
Then
 
 \begin{align*}
\mathcal{J}^s_p(u_C)\le \nu c_0-1\le \mathcal{J}^s_p(v)-1,
 \end{align*}
so there Lavrentiev phenomenon occurs for sufficiently large value of~$\nu$. 

Combination of properties (a) and (c) gives the sufficient condition on the parameters of the model:
\begin{align*}
\frac{p}{p-1}\bigg(s-\frac d p\bigg)< \frac{q}{q-1}\bigg(t-\frac {d+\alpha}{q}\bigg)
\end{align*}
sufficient for the Lavrentiev gap.
This proves the claim  and completes the proof the Theorem~\ref{thm:super}.
\end{proof}

We now  provide a straightforward embedding interpretation of the condition we have  derived to determine the existence of a Lavrentiev gap:

\begin{align}\label{condsup}
\frac{p}{p-1}\bigg(s-\frac d p\bigg)< \frac{q}{q-1}\bigg(t-\frac {d+\alpha}{q}\bigg).
\end{align}

Let~$\mathcal{C}=\mathcal{C}_\nu$ be the Cantor necklace with the Hausdorff dimension~$\frD:=m\frac{\log(2)}{ \log(\nu^{-1})}$, for the supercritical case it holds
\begin{align*}
 H^{s,p}(\Omega) \not\embedding W^{\frD,p}(\mathcal{C}),
\end{align*}
that is
\begin{align}\label{eq:12}
 \ind(H^{s,p}(\Omega))=s-\frac p d < \ind( W^{\frD,p}(\mathcal{C}))=\frD-\frac{\frD}{p}.
\end{align}

From the other side
\begin{align*}
 H^{t,q}_a(\Omega) \embedding W^{\frD,q}(\mathcal{C}),
\end{align*}
that is
\begin{align}\label{eq:22}
 \ind(H^{t,q}_a(\Omega))=s=t-\frac q d-\frac \alpha d \ge \ind( L^q(\mathcal{C}))=\frD-\frac{\frD}{q}.
\end{align}

Now combination of the inequalities~\eqref{eq:12} and~\eqref{eq:22} again gives the condition~\eqref{condsup}.

\begin{remark}
 Let us mention that related  question of density of smooth functions  poses a delicate challenge, even in the case of  local non-autonomous models with nonstandard growth.  Specifically,  the sharp condition for~$L^\infty$-truncation method for the local~$t,s=1$ double-phase and supercritical case~$p>d$  obtained by~\textcite{BulGwiSkr22} 
 \begin{align*}
  q\le p+\alpha\frac{p}{d}
 \end{align*} is different   condition that the one for the Lavrentiev gap to appear in the supercritical case from \textcite{BalDieSur20}
 \begin{align*}
  q> p+\alpha\frac{p-1}{d-1}.
 \end{align*}
This scenario may appear for nonlocal and mixed models as well.  
\end{remark}
%
%
%
%
%
%

\section{Nonlocal-local   mixed model \textbf{II}}\label{nonlocal}
In this section we construct example of Lavrentiev phenomenon for the nonlocal model \textbf{II},  so when the first phase~$\mathcal{J}_{p}^s(v)$ is nonlocal  and~$\mathcal{J}_{q,a}^t(v)$   is  local.  
\subsection{Subcritical case}
In the mixed  case we can use the estimate from the Lemma~\ref{lem:beh_com} for the competitor function to show that the first phase ~$\mathcal{J}^s_p(u_C)<\infty$.

The next step is to study how the local energy~$\mathcal{J}^1_{q,a}$ behaves on smooth functions. In particular, we will show that for~$v\in C^1(\overline \Omega)$ with~$v=u_D$ on~$\Omega^\complement$ the~$t,q$-energy is positive $\mathcal{J}^t_{q,a}(v)\ge c_0>0$. 

Let~$K^{\pm}$ denote the upper and lower cones:
\begin{align*}
K^+&:=\set{y\in \RRd: \abs{\overline{y}}\le \tfrac  14 y_d,\, 0<y_d\le 2},\\
K^-&:=\set{y\in \RRd: \abs{\overline{y}}\le -\tfrac  14 y_d,\, -2<y_d\le 0},\\
\btimes&:=K^+\cup K^-.
\end{align*}
 Define the restricted Riesz potential as
\begin{align*}
  \mathcal{I}_1^{\btimes}(v)(x)
  &:= \int_{\RRd}\frac{v(x+y)}{\abs{y}^{d-1}} \indicator_{\btimes}(y)\, dy.
\end{align*}

  Define $\mathcal I^+_1(v)$ and $\mathcal I^-_1(v)$ by
$$
    \mathcal{I}^\pm_1(v)(x):=\int_{K^{\pm}}\frac{v(x+y)}{\abs{y}^{d-1}} \, dy.
$$
  Then $\mathcal I_1^{\btimes} v(x)= \mathcal I_1^+ v(x)  + \mathcal I_1^- v(x)$.

Let us consider the case of one saddle point setup centered at zero.

\begin{lemma}\label{riesz_local}
For ~$u\in C^0(\overline{\Omega})$ with~$v=u_C$ on~$\partial\Omega$ and every~$x\in (-\frac 12, \frac 12)\times 0$ it  holds
\begin{align*}
1\lesssim \mathcal{I}_1^{\btimes}(\abs{\nabla v})(\overline{x},0).
\end{align*}
 \end{lemma}
 
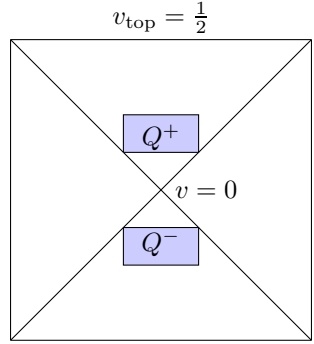
\begin{figure}[ht!]\centering \label{pic:riesz}
 \begin{tikzpicture}
 \coordinate (r0) at ( 0.0,  0.0) ; 
   \coordinate (s0) at (2.0, 2.0) ; 
   \coordinate (s1) at ( 2.0, -2.0); 
   \coordinate (si) at (-2.0, 2.0); 
   \coordinate (sq) at (-2.0, -2.0);
   \draw (r0)--(s0);
   \draw (r0)--(s1);
   \draw (r0)--(si);
   \draw (r0)--(sq);
   \draw(-2,-2)--(2,-2);
   \draw(-2,2)--(2,2);   
   \draw(2,-2)--(2,2);   
   \draw(-2,-2)--(-2,2);     
  \filldraw[fill=blue!20] (-0.5,0.5)--(0.5,0.5)--(0.5,1)--(-0.5,1)--(-0.5,0.5);
 \filldraw[fill=blue!20] (-0.5,-0.5)--(0.5,-0.5)--(0.5,-1)--(-0.5,-1)--(-0.5,-0.5);
 \node at (0,0.7) {$Q^+$};
   \node at (0,-0.7) {$Q^-$};
\node (U1) at (0, 2.3) {$v_{\text{top}}=\frac 12$};
   \node (U2) at (0.6, 0) {$v=0$};
    \end{tikzpicture}
    \caption{Riesz potential estimate, 1 saddle point.}
\end{figure}

\begin{proof}
  Since $v=u_C$ on~$\partial\Omega$ we have  at the top, so line~$x_d=1$, it equals to~$\frac 12$ and at the bottom equals to~$-\frac 12$. We denote  these values~$v_{\text{top}}$ and~$v_{\text{bot}}$ respectively and boxes~$Q^{\pm}_l:=\set{-\frac 14 l\le \overline x \le \frac 1 4 l; \pm l \le  x_d \le \pm 2l}$, with~$l\in(-\frac12,\frac 12)$ and~$\overline x\in (-1,1)^{d-1}$ of  the  the size~$l^d$, see Picture~\ref{pic:riesz}.  We first observe that
  \begin{align*}
    1=\frac 12-(-\frac 12)= \abs{v_{\text{top}} -v_{\text{bot}}}\le\abs{v_{\text{top}} -\mean{v}_{Q^+_l}} +
\abs{\mean{v}_{Q^+_l} -\mean{v}_{Q^-_l}} +\abs{\mean{v}_{Q^-_l} -v_{\text{bot}}},  \end{align*}
where~
\begin{align*}
\abs{\mean{v}_{Q^+_l}-\mean{v}_{Q^-_l}}\to 0
\end{align*}
since~$v\in C^0(\overline{\Omega})$.
Now
\begin{align*}
    \abs{v_{\text{top}} -\mean{u}_{Q^+_l}}\lesssim \mathcal{I}^+_1(\nabla v)(0).
\end{align*}
Indeed, by Jensen's inequality we get  
\begin{align*}
    \abs{v_{\text{top}} -\mean{u}_{Q^+_l}} &\lesssim \dashint_{Q^+_l} \int _{K^+} \frac{\abs{\nabla v (x+y)}}{\abs{y}^{d-1}}\, dydx\\
\\&\le \dashint_{Q_l^+} \int_{K^+} \frac{\abs{\nabla v(z)}}{\abs{z-x}^{d-1}}\, dzdx.     
\end{align*}
Here
\begin{align*}
\dashint_{Q_l^+} \frac{1}{\abs{z-x}^{d-1}}\, dz\le \abs{z}^{1-d}
\end{align*}
and therefore
\begin{align*}
    \abs{v_{\text{top}} -\mean{u}_{Q^+_l}} \lesssim \int_{K^+}\frac{\abs{\nabla v(z)}}{\abs{z}^{d-1}}\, dz=\mathcal{I}^+_1v(0).
\end{align*}

Analogously
\begin{align*}
    \abs{v_{\text{bot}} -\mean{u}_{Q^-_l}} &\lesssim \dashint_{Q^-_l} \int _{K^-} \frac{\abs{\nabla v (x+y)}}{\abs{y}^{d-1}}\, dydx\\
    &\lesssim \int_{K^-}\frac{\abs{\nabla v(z)}}{\abs{z}^{d-1}}\, dz=\mathcal{I}^-_1v(0).
\end{align*}
Now  we have
\end{proof}
We have:
\begin{enumerate}
  \item\label{i1} If $s-\frac{d-\frD}{p}<0$, then  $\mathcal{J}^s_p(u_C)<\infty$, so the competitor function~$u$ defined in \eqref{def:com} belongs to the energy space~$W_{u_D}$.
  \item\label{i2} $\mathcal{J}^t_{q,a}(u_C)=0$.
  \item\label{i3}  If $1-\frac{d-\frD}{q}+\frac{\alpha}{q}>0$, then   $\mathcal{J}^1_{q,a}(v)\ge c_0>0$ for all~$v\in C^0(\overline{\Omega})$  with~$v= u_D$ on~$ \Omega^{\complement}$.
 \end{enumerate}
 
 It remains to prove \ref{i3}. We use Lemma~\ref{riesz_local} and integrate over the Cantor necklace
  \begin{align*}
   0&<\frac 12 \int_{\frC} d\mu(x) \lesssim \int_{\frC}\mathcal{I}_1^{\btimes}(\nabla v)d\mu \\
& =\int_{\frC} \int_{\setR^d} \frac{\abs{\nabla v(z)}}{\abs{z}^{d-1}} \indicator_{\btimes} dz d\mu\ :=\mathrm{III},
   \end{align*}
and by the definition of the set~$K$ and Lemma~\ref{lem:cantor-estimates} we obtain
\begin{align*}
 \int_{\frC}  \indicator_{\btimes}d\overline y\lesssim \indicator_{\set{\abs{\overline y}\le\frac 14 (\abs{y_d})}} \abs{y_d}^{-d+1+\frD},
\end{align*}
and therefore
\begin{align*}
    \mathrm{III}\lesssim \int_{\setR^d}   \frac{\abs{\nabla v(x-y)}}{\abs{y}^{d-1-\frD}}\indicator_{\set{\abs{\overline y-\overline z}\le\frac 14 (\abs{y_d})}}dy.
\end{align*}
No by Hölder inequality we get 
\begin{align*}
    \mathrm{III} \lesssim \bigg(\int_\Omega\abs{y_d}^{\alpha}\abs{\nabla v}^q\, dy \bigg)^{\frac {1}{q}} \bigg(\int_\Omega   \abs{y_d}^{(-\alpha q^{-1}+\frD+1-d)q'}\indicator_{\set{\abs{\overline y-\overline z}\le\frac 14 (\abs{y_d})}}\, dy\bigg)^{\frac {1}{q'}}, 
\end{align*}
where the last integral is finite by Corollary~\ref{cor:cor} provided
\begin{align*}
 (\alpha q^{-1}-\frD-1+d)q'>-d+\frD,
\end{align*}
so
\begin{align*}
    1-\frac{d-\frD}{q}+\frac{\alpha}{q}>0.
\end{align*}

 Now, we use the argument with the scaling of the weight coefficient~$a(x)$. Let~$\nu>0$ be a free parameter  and define the energy with the scaled weight~$\nu a$ 
 \begin{align*}
    \mathcal{J}_\nu(v):= \mathcal{J}^s_p(v)+\mathcal{J}^1_{q,\nu a}.
 \end{align*}
Then
 \begin{align*}
    \mathcal{J}(v)\ge \mathcal{J}^1_{q,\nu a}(v)=\nu \mathcal{J}^t_{q,a}(v)\ge \nu c_0>0,
 \end{align*}
since~$v=u_D$ on~$ \Omega^{\complement}$. From the other side choose~$\nu$ such that
 \begin{align*}
    \mathcal{J}^s_p(u_C)\le \nu c_0-1.
 \end{align*}
Then
 
 \begin{align*}
\mathcal{J}^s_p(u_C)\le \nu c_0-1\le \mathcal{J}^s_p(v)-1,
 \end{align*}
so there Lavrentiev phenomenon occurs for sufficiently large value of~$\nu$. 

Combination of properties (a) and (c) gives the sufficient condition on the parameters of the model:

\begin{align*}
    s-\frac{d-\frD}{p}< 1-\frac{d-\frD}{q}+\frac{\alpha}{q},
\end{align*}
and since~$\frD\in (0,1)$ could be chosen arbitrary, that is equivalent to 
\begin{align*}
 q>sp+\alpha.
\end{align*}

\subsection{Supercritical case}
In the mixed  case we can use the estimate from the Lemma~\ref{lem:beh_com_super} for the competitor function to show that the first phase ~$\mathcal{J}^s_p(u_C)<\infty$.

We denote by~$x^+_{l,j}$ and $x^-_{l,j}$ the upper and lower  vertexes of an element of the necklace~$\mathcal{N}$, see the Figure 2.

The next step is to study how the energy~$\mathcal{J}^t_{q,a}$ behaves on  smooth functions. In particular, we will show that for~$v\in C^0(\overline \Omega)$, there holds $\mathcal{J}^t_{q,a}(v)\ge c_0>0$.
Similarly to the nonlocal case we get 
\begin{lemma}\label{lem:superRiesz}
 If~$v\in C^1(\overline{\Omega})$, then
 \begin{align}
  \abs{v(x^+_{l,j})-v(x^-_{l,j})}\lesssim \int_{\mathcal{N}_{l,j}}  \mathcal{I}_1^{\btimes}(\abs{\nabla v})(\overline{z},0) \indicator_{\mathcal{M}_8(y)} dz.
 \end{align}
\end{lemma}

This Riesz-type estimate and the scaling argument as in the Section 4 allows us to obtain the condition for the energy gap:
\begin{align*}
    \frac{p}{p-1}\bigg(s-\frac{d}{p}\bigg)<\frac{q}{q-1}\bigg(1-\frac{d+\alpha}{q}\bigg)
\end{align*}

\begin{remark}
 
It has been shown by \textcite{BuyLeeSong23} that, if~$1<p\le q<\infty$,~$0<s$,~$t=1$ and~$q\le sp+\alpha$, then every weak solution of the mixed nonlocal-local double phase problem~\textbf{II} in the subcritical case~$s-\frac d p<0$ is locally Hölder continuous. Thus, we can compare our condition 
for the Lavrentiev gap with  the one  for the absence of the Lavrentiev gap. This shows that our condition~$q>sp+\alpha$ is sharp in the subcritical case.
\end{remark}

\begin{remark}
 To  construct example of Lavrentiev phenomenon for the local-nonlocal  mixed model \textbf{III}, so for the case  when the first phase~$\mathcal{J}_{p}^s(v)$ is local  and~$\mathcal{J}_{q,a}^t(v)$   is  nonlocal, we use for both super and subcritical cases  the local  competitor function constructed in~\cite{BalDieSur20} and the Riesz potential estimates from the Section 4. 

\end{remark}

\bigskip
The author express her gratitude to Lars Diening and Moritz Kaßmann  for  the fruitful  discussions.

\printbibliography

\end{document}